\documentclass[12pt,twoside]{amsart}
\usepackage[latin1]{inputenc}
\usepackage{times}
\usepackage{amsthm}
\usepackage{amssymb, verbatim}
\usepackage{amsmath}
\usepackage{eucal}
\usepackage{mathrsfs}

\topmargin = - 40 pt 
\newcommand     {\Cset}    {{\mathbb C}}

\newcommand     {\A}    {{\mathcal A}}
\newcommand     {\B}    {{\mathcal B}}
\newcommand     {\C}    {{\mathcal C}}

\renewcommand   {\H}    {{\mathcal H}}

\newcommand     {\Q}    {{\mathcal Q}}

\newcommand {\vdim} {\text{\rm{dim}}}
\newcommand {\qdim} {\vdim_q}
\newcommand	{\hatG} {\widehat{G}}

\DeclareMathOperator    {\SU}  {SU}

\DeclareMathOperator	{\gkdim}  {GKdim}

\DeclareMathOperator{\Tr}{Tr}
\DeclareMathOperator{\mult}{mult}

\newtheorem{thm}{Theorem}[section]

\newtheorem{lemma}[thm]{Lemma}
\newtheorem{prop}[thm]{Proposition}
\newtheorem{cor}[thm]{Corollary}

\theoremstyle{definition}

\newtheorem{defn}[thm]{Definition}

\newtheorem{ex}[thm]{Example}

\newtheorem{rem}[thm]{Remark}

\numberwithin{equation}{section}

\begin{document}

\title[Dimension of CQG]
{Polynomial growth of discrete quantum groups, topological dimension of the dual and $^*$-regularity of the Fourier algebra}
\author[A.~D'Andrea]{Alessandro D'Andrea}
\email{dandrea@mat.uniroma1.it}
\author[C.~Pinzari]{Claudia Pinzari}
 \email{pinzari@mat.uniroma1.it}
\address{Dipartimento di Matematica, Universit\`a degli Studi di
Roma ``La Sapienza'', Roma}
\author[S.~Rossi]{Stefano Rossi}
\email{rossis@mat.uniroma2.it}
\address{Dipartimento di Matematica, Universit\`a degli Studi di
Roma ``Tor Vergata''}

\begin{abstract}
Banica and Vergnioux have shown that the dual discrete quantum group of a compact simply connected Lie group has polynomial growth of order the real manifold dimension. We extend this result to a general compact group and its topological dimension, by connecting it with    the Gelfand-Kirillov dimension of an algebra. Furthermore, we show that polynomial growth for a compact quantum group $G$ of Kac type implies $^*$--regularity of the Fourier algebra $A(G)$, that is every closed ideal of $C(G)$ has a dense intersection with $A(G)$. In particular, $A(G)$ has a unique $C^*$--norm.
\end{abstract}

\maketitle
\centerline{\it Dedicated to the memory of John E. Roberts}\bigskip
\section{Introduction}

The notion of polynomial growth for the dual of a compact quantum group $G$ of Kac type was introduced by Vergnioux \cite{Vergnioux}. It turns out to be a growth property of the vector dimension function  $\dim$ on the associated representation ring $R(G)$, and, as such, it can be generalised to the non-Kac type case.
It has since attracted much interest, one of the reasons being that   it is a natural generalisation of the corresponding classical notion for discrete groups.

In particular, Banica and Vergnioux have shown in \cite{BV, Vergnioux} that if $G$ is a connected, simply connected, compact Lie group then $(R(G), \dim)$ has polynomial growth, and the order of the growth equals the manifold dimension of $G$. 
This suggests that  polynomial growth of $(R(G), \dim)$ may be understood as a non-commutative analogue of the topological dimension.

One of the aims of this note is to support this viewpoint. We first connect growth of a quantum group with the more general notion of growth of an algebra in the sense introduced by Gelfand and Kirillov and show that this leads to an extension of Banica and Vergnioux theorem to all compact groups.
 
More specifically, we recall that Gelfand and Kirillov, motivated by the isomorphism problem of Weyl division algebras, introduced what is now called the GK dimension of an algebra $\A$ \cite{GK}. The GK dimension measures the best polynomial growth rate of $\A$.
By definition, every algebra of polynomial growth is the inductive limit of finitely generated algebras of finite GK dimension.

Gelfand-Kirillov dimension equals the Krull dimension for finitely generated commutative algebras. For commutative domains, it further equals the transcendence degree of the corresponding fraction field over the base field. 
  
We remark that if $G$ is a compact quantum group, the GK dimension of the canonical dense Hopf algebra equals the growth rate of the vector dimension function in the sense of Banica and Vergnioux. We further show that in the classical case, the GK dimension also equals the Lebesgue topological dimension, and, as mentioned above, this extends Banica and Vergnioux's theorem to general compact groups. The further connection of the GK dimension with the classical transcendence degree also recovers a theorem of Takahashi on the topological dimension of a compact group \cite{Takahashi}, in turn extending 
a classical result of Pontryagin for compact abelian groups on the equality between dimension and rank of the dual group.
  
The GK dimension of the group algebra of a discrete group of polynomial growth coincides with the Bass rank of the group, by well-known results of Bass and Gromov \cite{Krause_Lenagan}.

Motivated by these facts, we introduce the {\it topological dimension} of a compact quantum group $G$ as the Gelfand-Kirillov dimension of the dense Hopf algebra.
For example, the compact quantum groups with the same representation theory as that of a given compact Lie group, in the sense of \cite{NY}, have finite topological dimension. 
Saying that ($R(G), \dim$) has polynomial growth means precisely that the associated Hopf $C^*$--algebra is the inductive limit of Hopf $C^*$--algebras of compact matrix quantum groups of finite topological dimension.

In the second part of the paper we   discuss a structural consequence of  polynomial growth.
It is known that compact quantum groups $G$ of subexponential growth are coamenable \cite{Banica99, BV}.
However, we gain further information in the case of polynomial growth.

More specifically, the Fourier algebra $A(G)$ of a compact quantum group is known to be a Banach algebra. If we moreover assume that $G$ is of Kac type, $A(G)$ becomes an 
 {\it involutive} Banach algebra with involution obtained extending the involution of the canonical dense Hopf subalgebra ${\mathcal Q}_G$. The maximal  completion $C_{\text{max}}(G)=C^*({\mathcal Q}_G)$  can also be regarded as   the enveloping $C^*$--algebra of    $A(G)$. There is a natural surjective map between the primitive ideal spaces $\Psi: \text{Prim}{C_{\text{max}}(G)}\to\text{Prim}{A(G)}$ which is continuous  when either space is endowed with the hull-kernel topology. We show that if $G$ is of Kac type and  of polynomial growth, $\Psi$ is a homeomorphism. This in particular implies that   the Jacobson topology on $\widehat{C_{\text{max}}(G)}$ coincides with the Jacobson topology induced by the Fourier subalgebra $A(G)$.

In the framework of Banach $^*$--algebras, this property is known as $^*$--regularity. 
If $A$ is such an algebra, an equivalent statement of $^*$--regularity is that every closed ideal of $C^*(A)$,   has dense intersection with $A$. As an   important general consequence of this property we have  that   $A$ has a unique $C^*$--norm \cite{Barnes81, Barnes83,   Boidol84}. This in particular applies to $A(G)$, as opposed to the canonical dense Hopf algebra which has many $C^*$--norms already for $G={\mathbb T}$. The property of 
$^*$--regularity has been first shown for the pair of the
group algebra   $L^1(\Gamma)$ of a locally compact group with Haar measure of polynomial growth and its
 $C^*$--envelope $C^*(\Gamma)$  \cite{BLSV}. In this sense,   our result is a discrete quantum analogue.

To conclude, it is perhaps worth emphasizing that our quantum group approach provides a new   look at $^*$--regularity   as a geometric property, in that it may now be   interpreted  as  a regularity condition of the spectrum of $C(G)$ enjoyed by the subclass of those (Kac-type) $G$ which can be approximated 
by  quotients of finite  topological dimension.

The paper is organised as follows. Sect. \ref{prel} is dedicated to preliminaries on GK dimension and compact quantum groups. In Sect. \ref{banver} we compute GK dimension for compact groups, derive the generalised  Banica-Vergnioux theorem, and introduce the notion of topological dimension for a compact quantum group. Its basic properties are discussed in Sect. \ref{basic}.
Sect. \ref{fourier} deals with the Fourier algebra of a compact quantum group. We describe  an approach in terms of representations which parallels the classical treatment in the book by Hewitt and Ross \cite{HR}. In Sect. \ref{growth} we show our $^*$--regularity result, and finally in Sect. \ref{nonkac} we comment on the problem of  extending $^*$--regularity to the non-Kac-type case, that we shall hopefully study elsewhere.
 
\section{Preliminaries on the GK dimension and compact quantum groups}\label{prel}

\subsection{Algebras of polynomial growth and GK dimension}\label{PolyGK}
Let $\A$ be a unital algebra over a field $k$. 
We recall the definition of GK dimension of $\A$ \cite{GK}.
Let $V$ be a subspace of $\A$ and form the subspace of elements that can be written as sums of products of at most $n$ elements of $V$,
$$V_n=\Sigma_{k=0}^n V^k.$$ 
 
\begin{defn}
We shall say that $\A$ has {\em polynomial growth} if for every finite dimensional subspace $V$ of $\A$ there is $\gamma\in{\mathbb R}^+$ such that $$\dim(V_n)=O(n^\gamma).$$
\end{defn}
If $\A$ has polynomial growth one can compute the infimum of polynomial exponents $\gamma$ associated to $V$ with the formula
$$\inf\{\gamma>0: \dim(V_n)=O(n^\gamma) \}=\overline{\lim}_{n\to\infty}\frac{\log \dim(V_n)}{\log(n)}.$$
The GK dimension of $\A$ is defined by
$$\gkdim(\A)=\sup_V\overline{\lim}_{n\to\infty}\frac{\log \dim(V_n)}{\log n}.$$
It is easy to see that $\dim(V_n)$ grows at most exponentially and that $V$ generates a finite-dimensional algebra if and only if $\dim(V_n)=\dim(V_{n+1})$ for some $n$. Hence either $\dim(V_n)$ is eventually constant or $\dim(V_n)\geq n+1$. Thus $\gkdim(\A)$ takes no value in the interval $(0, 1)$, and $\gkdim(\A)=0$ if and only if $\A$ is inductive limit of finite dimensional subalgebras.  It is also known that no real number in $(1,2)$  either can arise as the GK dimension of an algebra, hence $\gkdim(\A)$ is either $0$, $1$ or $\geq2$ \cite{Krause_Lenagan}. GK dimension seems to take integral or infinite values on all known algebras admitting a Hopf algebra structure \cite{Zhuang}.

In the case where $\A$ is a finitely generated algebra, polynomial growth takes a simpler form. First, it suffices to verify it only on a finite dimensional subspace $V$ generating $\A$ as an algebra. Indeed, any other finite dimensional subspace $W$ is contained in some $V_s$ and therefore  $W_n\subset V_{sn}$ for all $n$. If we know that $\dim(V_n)=O(n^\gamma)$ then $\dim(W_n)=O(n^\gamma)$ as well and the growth exponent $\overline{\lim}_{n\to\infty}\frac{\log \dim(W_n)}{\log n}$ of $W$ is bounded above by that of  $V$; consequently, the growth exponent does not depend on the choice of the generating subspace and equals $\gkdim(\A)$. This also shows that every finitely generated algebra of polynomial growth has finite $\gkdim$.
Notice that if $\A$ is the inductive limit of subalgebras $\A_\gamma$ then 
$\gkdim(\A)=\lim_\gamma\gkdim(\A_\gamma).$
In general, a polynomial growth algebra is the inductive limit of finitely generated algebras with finite GK dimension.

In fact, for some applications, we will use the following stronger notion of polynomial growth.
\begin{defn}
Let $\A$ be a finitely generated algebra. We shall call $\A$ of {\it strict polynomial growth of degree $\gamma\in{\mathbb R}^+$} if there is a f.d. generating subspace $V$ and constants $c$, $d>0$ such that eventually
$$cn^\gamma\leq \dim(V_n)\leq dn^\gamma.$$
\end{defn}

Obviously,
if $\A$ has strict polynomial growth of degree $\gamma$ then   $\gkdim(\A)=\gamma$. 

\begin{prop}
If strict polynomial growth of degree $\gamma$ holds for a generating subspace $V$ then it holds for all other generating subspaces. 
\end{prop}
\begin{proof}
We have already shown independence of the right inequality. For the left inequality, let, as before, $W$ be another generating subspace. There is a positive integer $t$ such that $W_{tm}\supset V_m$ for all $m$.
For a fixed $n\in{\mathbb N}$ choose $m$ such that $tm\leq n<(m+1)t$. Then $W_n\supset W_{tm}\supset V_m$. Hence $\dim(W_n)\geq cm^\gamma\geq\frac{c}{t^\gamma}(n-t)^\gamma\geq c' n^\gamma$.
\end{proof}

If $\A$ is a commutative algebra, its GK dimension reduces to   classical notions. More precisely, we recall the following result.  
\begin{thm}\label{gkcommutative}
If $\A$ is a commutative unital $k$-algebra then $\A$ is of polynomial growth and $\gkdim(\A)$ is either a non-negative integer or infinite. More precisely,
\begin{itemize}
\item[{\rm a)}]
if $\A$ is finitely generated, $\gkdim(\A)$ is finite and equals the Krull dimension $d$ of $\A$. In fact, $\A$ has strict polynomial growth.
\item[{\rm b)}]
If $\A$ has no zero divisors then $\gkdim(\A)=\text{tr.deg}(Q(\A))$, the transcendence degree of the fraction field of  $\A$ over $k$.
\end{itemize}
\end{thm}
\begin{proof} A possible reference for this theorem is \cite[Chapter 4]{Krause_Lenagan}. More precisely, b) follows from Cor. 4.4 and Prop. 4.2, while a), except the property of strict polynomial growth, is stated in  Theorem 4.5. For the last property, notice that the proof of Lemma 4.3 of the same book shows that if $B\subset A$ is an inclusion of finitely generated commutative algebras such that $A$ is finitely generated as a $B$-module then $\dim(V_n)\leq r\dim(W_{2n-1})$ where $W$ and $V$ are generating subspaces of $B$ and $A$ respectively, and $V$ contains a set of generators of $A$ as a $B$-module, whose cardinality is denoted by $r$. Assuming in addition that $V$ contains $W$ as well, we also gain $\dim(V_n)\geq\dim(W_n)$.
Thus if $B$ has strict polynomial growth then so does $A$ and with the same degree. As in the proof of \cite[Theorem 4.5]{Krause_Lenagan}, we may now appeal to Noether's normalisation lemma and choose for $B$ the polynomial algebra $k[x_1,\dots, x_d]$.
\end{proof}

Another important class of examples arises from discrete groups.  Polynomial growth of the group algebra ${\mathbb C}\Gamma$ of a finitely generated discrete group $\Gamma$ reduces to the usual notion of polynomial growth for $\Gamma$. It is well known that groups of polynomial growth have been studied, among others, by Bass, Milnor, Wolf and Gromov. In particular, Bass and Wolf showed that nilpotent groups have strict polynomial growth of degree given by the Bass rank, and Gromov proved that every polynomial growth group is virtually nilpotent. See \cite{Krause_Lenagan} for references. 
\bigskip

\subsection{Compact quantum groups}
We briefly  recall the  notion of a compact quantum group along with the main properties, as developed by Woronowicz \cite{WoroCMP}, see also \cite{MaesVanDaele,NTbook}.

A compact quantum group is defined by a pair $G=(Q, \Delta)$ where $Q$ is a unital $C^*$--algebra and $\Delta$ is a coassociative unital $^*$--homomorphism $$\Delta: Q\to Q\otimes Q$$ such that $I\otimes Q\Delta(Q)$ and $Q\otimes I\Delta(Q)$ are dense in the minimal tensor product $Q\otimes Q$. 

The basic example is given by the algebra $C(G)$ of continuous functions on a compact group, and every commutative example is of this form. It is customary to keep the same notation $C(G)$ for $Q$ also when $Q$ is not commutative and we shall occasionally follow this convention.
 
Another important class of examples is provided by discrete groups. If $\Gamma$ is such a group then the group $C^*$--algebra
$C^*(\Gamma)$, which is the completion of the group algebra ${\mathbb C}\Gamma$ in the maximal $C^*$--norm, becomes a compact quantum group with coproduct $\Delta(\gamma)=\gamma\otimes\gamma$, $\gamma\in \Gamma$. We may also consider the reduced $C^*$--completion $C^*_{\text{red}}(\Gamma)$ and still obtain a compact quantum group.
We shall refer to these as cocommutative examples, since the coproduct is invariant under the automorphism that exchanges the factors of $C^*(\Gamma)\otimes C^*(\Gamma)$. Furthermore, every cocommutative compact quantum group can be obtained as the completion of ${\mathbb C}\Gamma$ with respect to some $C^*$--norm, which is bounded by the reduced and the maximal norm. This fact extends to general compact quantum groups where elements of $\Gamma$ are replaced by the matrix coefficients of representations of $G$, that we recall next, and is a consequence of Woronowicz density theorem. The quantum group $G$ is called coamenable if the reduced and maximal norm coincide.
  
A {\em representation} of $G$ can be defined as a unitary element $u\in{\B}(H)\otimes Q$, where $H$ is a finite dimensional Hilbert space, satisfying $\Delta(u_{\xi, \eta})=\sum_r u_{\xi, e_r}\otimes u_{e_r,\eta}$, where  $u_{\xi, \eta}$, the {\em matrix coefficients} of $u$, are defined by $u_{\xi, \eta}=(\xi^*\otimes 1) u (\eta\otimes 1)$, with $\xi$ and $\eta$ vectors of $H$ here regarded as   operators $\Cset\to H$ between Hilbert spaces, and $(e_r)$ is an orthonormal basis of $H$. The more general notion of invertible representation is meaningful, and in fact invertible representations arise naturally in the construction of the conjugate representation, that we next recall. However every invertible representation turns out to be equivalent to a unitary one.
Henceforth the term representation will always mean a unitary representation on a finite dimensional Hilbert space. 

An intertwiner between two   representations $u$ and $u'$ is a linear operator $T$ from the space of $u$ to that of $u'$   such that $(T\otimes I) u=u' (T\otimes I)$. Two representations are equivalent if there is an invertible intertwiner, which can always be chosen to be unitary. 

The category whose objects are representations of $G$ and whose arrows are intertwiners is a tensor $C^*$--category with conjugates in the sense of, e.g., \cite{NTbook}. Subrepresentations, direct sums of representations as well as irreducible representations are defined in the natural way. The tensor product  $u\otimes u'$ of two representations acts on the tensor product Hilbert space, and is determined by $$(u\otimes u')_{\xi\otimes\xi', \eta\otimes\eta'}=u_{\xi,\eta}u'_{\xi',\eta'}.$$ Furthermore,  the conjugate $\overline{u}$ of any representation $u$ is determined, up to unitary equivalence, by an invertible antilinear operator $j$ from the space of $u$ to that of $\overline{u}$ satisfying
$$\overline{u}_{\phi,\psi}=(u_{j^{-1}\phi, \ j^*\psi})^*.$$ It follows that $$R=\sum_r je_r\otimes e_r\in(\iota, \overline{u}\otimes u), \quad \overline{R}=\sum_s j^{-1}f_s\otimes f_s\in (\iota, u\otimes\overline{u}),$$ with $\iota$ the trivial representation. 
Every representations can be decomposed as a direct sum of irreducible representations, in a  unique way up to equivalence. 

We shall denote by $\hatG$ a fixed complete set of inequivalent irreducible representations. Notice that later on we shall use the same symbol for the discrete quantum group dual to $G$, but this should not cause confusion.

The linear span $\Q_G$ of matrix coefficients of representations is a canonical dense $^*$--subalgebra of $Q_G$, which has the structure of a Hopf $^*$--algebra \cite{WoroCMP, cqg}. The collection $\{v_{r s}, v\in\hatG\}$ of all matrix coefficients with respect to a choice of orthonormal bases is linearly independent and spans $\Q_G$.

In the classical case, representations describe usual unitary representations and $\Q_G$ is the Hopf algebra of representative functions on $G$. If $G$ is cocommutative and arises from $\Gamma$ then every element of $\Gamma$ is a one-dimensional representation, and these are the only irreducible representations. 

Most importantly, $Q_G$ has a unique {\em Haar state} $h$, which means that $h$ is a state satisfying the invariance condition $h\otimes 1(\Delta(a))=h(a)I=1\otimes h(\Delta(a))$ for all $a\in Q_G$. It is  determined by requiring that it annihilate all coefficients of non-trivial irreducible representations.

For any irreducible $u$ with conjugate defined by $j_u$, set  $F_u=j_u^*j_u$. This operator depends on the choice of $\overline{u}$ and $j_u$ only up to a positive scalar factor; we will henceforth normalise the choice of $j_u$ so that $\Tr(F_u)=\Tr(F_u^{-1})$, which yields a positive invertible operator $F_u$ canonically associated with $u$. The scalar $\qdim(u)=\Tr(F_u)\geq \dim(u)$ is the {\em quantum dimension} of $u$. The quantum group $G$ is called of Kac type if $F_u=I$ for all $u$, which is equivalent to $h$ being a trace or to $\qdim(u)=\dim(u)$ holding for all representations.
 
The Haar state satisfies the following {\em orthogonality relations} for matrix coefficients of irreducible representations, see \cite{WoroCMP},
\begin{equation}\label{orth1}
h(v_{\psi,\phi}^*u_{\xi,\eta})=\delta_{v,u} \frac{1}{\qdim(u)}\langle\xi, F_u\psi\rangle\langle\phi,\eta\rangle=\delta_{v,u} \frac{1}{\qdim(u)}\Tr(F_u\Theta^v_{\psi,\phi}\Theta^u_{\eta, \xi})
\end{equation}
where $\Theta^u_{\eta,\xi}$ is the rank 1 operator $\Theta^u_{\eta,\xi}(\zeta)=\langle \xi,\zeta\rangle\eta$ and $\Tr$ is the non-normalised trace. Similarly
\begin{equation}\label{orth2}
h(u_{\xi,\eta}v_{\psi,\phi}^*)=\delta_{v,u}\frac{1}{\qdim(u)}\langle\xi,\psi\rangle\langle\phi,F_u^{-1}\eta\rangle=\delta_{v,u}\frac{1}{\qdim(u)}\Tr(F_u^{-1}\Theta^u_{\eta,\xi}\Theta^v_{\psi,\phi}).
\end{equation}
  
\section{A theorem of Banica and Vergnioux}\label{banver}

The following observation makes Theorem \ref{gkcommutative} more precise for function algebras of compact groups.

\begin{thm}\label{corollary}
If  $G$ is a compact group then  
\begin{itemize}
\item[{\rm a)}]
$\gkdim(\Q_G)$ equals the Lebesgue topological dimension of $G$,
\item[{\rm b)}]
if $G$ is a Lie group then $\gkdim(\Q_G)$ equals the dimension of $G$ as a real manifold.
\end{itemize}
\end{thm}
\begin{proof}
$\Q_G$ is inductive limit of finitely generated algebras as $G$ is  the inverse limit of compact Lie groups. Lebesgue dimension commutes with inverse limits of compact Hausdorff spaces, while GK dimension commutes with inductive limits of algebras.
Also, for a compact Lie group, the Lebesgue dimension coincides with the real manifold dimension. These remarks show that a) follows from b).
  
In order to show b), assume $G$ is a compact Lie group. The complexification $G_{\mathbb C}$ of $G$ is an algebraic group and $\Q_G$ identifies with the coordinate ring ${\mathcal O}(G_{\mathbb C})$ of $G$, see e.g., \cite{BTD}.
The real dimension of $G$ equals the complex dimension of the Lie group $G_{\mathbb C}$ as the Lie algebra of $G_{\mathbb C}$ is the complexification of the Lie algebra of $G$.
Now, the complex dimension of the Lie group $G_{\mathbb C}$ equals the Krull dimension of ${\mathcal O}(G_{\mathbb C})$, which is a finitely generated complex commutative algebra. By property a) of Theorem \ref{gkcommutative}, the latter coincides with the GK dimension.

We also give an alternative argument. By \cite[Theorem A]{Takahashi} the topological dimension of a general compact group $G$ is given by the transcendence degree of $\Q_G$ over ${\mathbb C}$. Assume for simplicity that $G$ is connected. Then $\Q_G$ has no zero divisors, hence the latter equals $\gkdim(\Q_G)$ by Theorem \ref{gkcommutative} b).
\end{proof}

\begin{defn} Let $G$ be a compact quantum group whose associated dense Hopf algebra $\Q_G$ is of polynomial growth. We define the \em{topological dimension} of $G$ by $$\dim(G)=\gkdim(\Q_G).$$   \end{defn}

It is an easy but important remark that computation of $\dim(G)$ can  be spelled out in terms of representations for all compact quantum groups, and in fact this connects it with the work of Banica and Vergnioux \cite{Vergnioux, BV}. We recall their main definitions. We
pick a dimension function
$$d: R(G)^+\to{\mathbb R}^+$$ on the representation ring of $G$ and define, for any representation $u$ of $G$ and any positive integer $n$, the sequence
$$b(u, n):=\sum d(v)^2.$$
where the sum is taken over irreducible subrepresentations $v\subset u^{\otimes k}$, for $k\leq n$.
One can then introduce a notion of polynomial (resp., subexponential, exponential) growth for $d$ requiring that for any $u\in R(G)^+$,
$b(u, n)=O(n^\gamma), \text{ for some } \gamma>0,$ (resp., ${\lim}_{n\to\infty} b(u, n)^{1/n}=1$, ${\lim}_{n\to\infty} b(u, n)^{1/n}<1$. Notice that these limits always exist.) We shall be interested in the growth of the following dimension functions:
\begin{itemize}
\item[{\rm a)}]
$d=\vdim$, associating every representation with its vector space dimension, and referred to as the vector dimension function,\medskip

\item[{\rm b)}]
$d=\qdim$, the quantum dimension.
\end{itemize}

\begin{lemma}\label{computegk}
Let $u$ be a  representation of a compact quantum group $G$ and $V$ be the linear span of the matrix coefficients of $u$. Then $V_n$
is the linear span of matrix coefficients of the set of inequivalent irreducible subrepresentations $v$ of $u^{\otimes k}$, for $k\leq n$, and we have equality
$$b(u, n)=\dim(V_n),$$
where the left hand side refers to the vector dimension function $\dim$.
\end{lemma}
\begin{proof}
$V_n$ is the linear span of products of matrix coefficients of $u$ up to length $n$. But finite products of entries of $u$ are coefficients of tensor powers of $u$, hence complete reducibility shows that $V_n$ is as stated. Computation of dimension follows from linear independence of coefficients of irreducible representations in $\Q_G$.
\end{proof}

\begin{ex}
Let $u_n$ denote the irreducible representation of $\SU(2)$ of dimension $n+1$. Then $u_1$ is a generating representation, and the family of irreducible subrepresentations of $u_1^{\otimes n}$ consists precisely of all $u_i$ such that $i\leq n$ has the same parity as $n$. Then vector and quantum dimension coincide and the associated sequence is 
$$b(u_1, n) = \frac{(n+1)(n+2)(2n+3)}{6},$$
hence $\dim(\SU(2)) = 3$.
\end{ex}

\begin{ex} More generally, the growth of $\dim$ for $G_q$ is the same as that for $G$, which is polynomial of degree equal to the manifold dimension of $G$, by \cite[Theorem 2.1]{BV}. However, the growth of   $\qdim$ is exponential. For $A_o(F)$ $\dim$ (hence, $\qdim$) has exponential growth as soon as $F$ is a matrix of order at least $3$.
\end{ex}

\begin{prop}\label{caratterization} Let $G$ be a compact quantum group. The following properties are equivalent:
\begin{itemize}
\item[{\rm a)}] 
$\Q_G$ has polynomial growth,
\item[{\rm b)}] 
$\text{\rm dim}$ has polynomial growth,
\item[{\rm c)}] 
$C(G)$ is the inductive limit of Hopf $C^*$--algebras of compact matrix quantum groups of finite 
topological dimension. 
\end{itemize}
\end{prop}

Theorem \ref{corollary}, along with the above proposition 
allow us to generalise \cite[Theorem 2.1]{BV} from connected, simply connected, compact Lie groups to all compact Lie groups. 

\begin{cor}\label{bv}
Let $G$ be a compact Lie group, $u$ a selfadjoint generating representation and let $N$ be the dimension of $G$ as a real manifold. Then the sequence $b(u, n)$ has strict polynomial growth, in that it is bounded above and below by a polynomial of degree $N$.
\end{cor}

\section{Basic properties}\label{basic}

\begin{prop} 
If $G$ is a compact quantum group such that $\Q_G$ is of polynomial 
or subexponential growth and $H$ is either a quotient or a subgroup of $G$, then $\Q_H$ is accordingly of polynomial or subexponential growth. In the first case, $\dim H \leq \dim G$.
\end{prop} 
\begin{proof} A subalgebra or a quotient algebra $\B$ of an algebra $\A$ of polynomial or subexponential growth has the same property, and it is easy to see that $\gkdim(\B)\leq \gkdim(\A)$ in the first case. On the other hand, quotients and subgroups of $G$ are described respectively by subalgebras and quotient algebras of $\Q_G$.\end{proof}

An explicit proof in terms of growth of the sequences $b(u, n)$ can alternatively be worked out. One can indeed bound the sequences corresponding to quotients or subgroups by those related to $G$.

\begin{rem}
Notice that even though subquotients of compact quantum groups having a generating representation may fail to have a generating representation, finite topological dimension is inherited by all subquotients.
\end{rem}

We next give a simple result which guarantees polynomial growth.

\begin{prop}\label{commutative}
Let $G$ be a compact quantum group with commutative representation ring $R(G)$. Assume that for any irreducible representation $v$,
$\dim(V_n)=O(n^\gamma)$ for some $\gamma>0$, where $V$ is the linear span of coefficients of $v$. Then 
$\Q_G$ is of polynomial growth. 
\end{prop}

\begin{proof}
Pick a subspace $W\subset\Q_G$ spanned by coefficients of a representation $u$ of $G$,  and decompose $u$ into its irreducible components $u=v_1\oplus\dots\oplus v_p$. Denote by $V_{i}$ the span of coefficients of $v_i$. The subspace $W_n$, being the linear span of coefficients of powers of $u$,    is already spanned  by $V_{1}^{k_1}\dots V_{p}^{k_p}$ with $k_1+\dots+ k_p\leq n$ thanks to commutativity of $R(G)$.
Hence $\dim(W_n)\leq {n+p \choose p}\dim((V_{1})_n)\dots \dim((V_{p})_n)$ showing that $\dim(W_n)$ is bounded by a polynomial.\end{proof}

We have seen that the property of polynomial growth for $\Q_G$ is in fact a property of the vector dimension function $\dim$. On the other hand, $R(G)^+$ is also endowed with the quantum dimension function $u\mapsto\qdim(u)$, which in general exceeds $\dim$. The stronger property of polynomial growth of $\qdim$ was introduced, among other things, by Vergnioux \cite{Vergnioux}.  
Later on we will be interested in quantum groups $G$ with this property.
 
\begin{prop}\label{unique} Let $d$ be a dimension function on $R(G)$ of subexponential growth. Then any other dimension function $d'$ on $R(G)$ satisfies $d'(u)\geq d(u)$ for all $u$. In particular, $R(G)$ admits at most one dimension function of subexponential growth. 
\end{prop}
\begin{proof}
Assume that for some $u$, $d' = d'(u) < d(u) = d$. If $v \subset u^{\otimes k}$, $k\leq n$, is an irreducible subrepresentation then $d(v) \leq \sqrt{b(u, n)}$, where obviously $b(u,n)$ is associated to $d$. We know that $d(u^k) = d^k$ for $k\leq n$.
Hence there are at least $(1+d+\dots+d^n)/\sqrt{b(u, n)}$ irreducible summands (counted with multiplicity) in an irreducible decomposition of $\oplus_{k=0}^n u^{\otimes k}$.
Let us compare the $d'$-dimension using the same decomposition. We have that $d'(u^k)$ equals $d'^k$ and that each irreducible summand has $d'$-dimension at least $1$. Consequently,
$$(1+\dots+d'^n) \geq (1+\dots+ d^n)/\sqrt{b(u, n)},$$
implying $b(u, n) \geq \frac{1}{(n+1)^2} (d/d')^{2n}$, which   contradicts subexponentiality of $b(u, n)$.
\end{proof}

Notice   that $d(u)$ may differ from the vector dimension function.
For example, for any integer $n\geq 2$ if $F\in M_n({\mathbb C})$ satisfies suitable properties then $\SU_q(2)$ and $A_o(F)$ have isomorphic representation rings (in fact isomorphic representation categories), hence the vector dimension function of $R(\SU_q(2))^+$ gives a dimension function $d$ on $R(A_o(F))^+$ of polynomial growth smaller than its vector dimension function.

\begin{cor}\label{vergnioux}
The following properties are equivalent for a compact quantum group $G$.
\begin{itemize}
\item[{\rm a)}] $u\mapsto \text{\rm dim}(u)$ has polynomial (subexponential) growth and $G$ is of Kac type,
\item[{\rm b)}] $u\mapsto\text{\rm dim}_q(u)$ has polynomial (subexponential) growth.

\end{itemize}
\end{cor}

\begin{proof} If b) holds then $\dim$ must have subexponential growth since it is bounded above by $\qdim$.
Hence $\qdim=\dim$ by the previous proposition, and this shows that $G$ is of Kac type. The converse is obvious.
\end{proof}

We remark that Prop. \ref{unique} and Cor. \ref{vergnioux} are not new, but the original arguments are scattered in the literature and sometimes not explicitly stated, our proofs being perhaps more direct. For example, Cor. \ref{vergnioux} follows from \cite[Prop. 4.4]{Vergnioux}, and \cite{Banica99, BV}.

Also, Prop. \ref{unique}, follows from the following facts. The vector dimension function of a coamenable compact quantum group is minimal among all dimension functions, see \cite{NTbook}. Furthermore, the following relationship between subexponential growth and coamenability has been highlighted in \cite{Banica99, BV}, where the main focus was on compact quantum groups of Kac type. For the reader's convenience, we complete details of the proof to point out that the Kac assumption is not needed.

\begin{thm}
Every compact   quantum group $G$ for which $\Q_G$ is of subexponential growth is coamenable.
\end{thm}
\begin{proof}
Let $u=\bar{u}$ be a self-conjugate representation of $G$. Thanks to a characterisation of coamenability given by Skandalis, cf. \cite[Theorem 6.1]{Banica99}, see also \cite[Theorem 2.7.10]{NTbook} and the remark following it, we have to show that $\|\chi_u\|_r={\rm dim}(u)$. The inequality $\|\chi_u\|_r\leq{\rm dim}(u)$ is always trivially verified. Therefore, it suffices to prove the reverse   inequality.
We start by observing that  $\|\chi_u\|_r=\|\chi_u^{2k}\|_r^{\frac{1}{2k}}$ for every natural number $k$, since $\chi_u$ is self-adjoint.
We have $\|\chi_u^{2k}\|_r^{\frac{1}{2k}}\geq h(\chi_u^{2k}\chi_u^{2k})^{\frac{1}{4k}}=h(\chi_{u^{\otimes^{2k}}}\chi_{u^{\otimes^{2k}}})^{\frac{1}{4k}}\geq m_{2k}(1)^{\frac{1}{2k}}$, 
where $m_{2k}(1)$ is the multiplicity of the trivial representation in $u^{\otimes^{2k}}$ and the last inequality follows from  a  decomposition of $u^{\otimes^{2k}}$ into irreducible components. 

The statement is now a consequence of the proof of \cite[Prop. 2.1]{BV}, which shows that subexponential growth is enough to establish $\limsup_{k\rightarrow\infty} m_{2k}(1)^{\frac{1}{2k}}\geq {\rm dim}(u)$.
\end{proof}

We conclude by observing that most of the results of this and the previous section can be extended to ergodic actions $\delta: C\to C\otimes Q$ of compact quantum groups on unital $C^*$-algebras.
In the commutative case, $C$ will be the algebra of continuous functions on a quotient space $G/K$ by a closed subgroup, and one can generalise Theorem \ref{corollary} to this setting. In the general case, ergodic actions still enjoy a good spectral theory \cite{PR, Podles}, which can be used to extend Lemma \ref{computegk}. Indeed,    $C$ has a canonical dense $^*$--subalgebra $\C$ linearly spanned by a choice of elements carrying irreducible representations under the action. For a given (reducible) representation $u$, $V$ is the subspace of $\C$ corresponding to the irreducible components of $u$. Thus $V_n$ becomes the linear span of elements of $\C$ corresponding to set $S_{u, n}$ of irreducible and spectral subrepresentations of $u^{\otimes k}$, for $k\leq n$. The computation of now yields $\dim(V_n)=\sum_{v\in S_{u,n}}\dim(v)\mult(v)$, where $\mult(v)$ is the multiplicity of $v$ in the action.
Unlike the classical case, examples are known of ergodic actions of $\SU_q(2)$ for which $\mult(v)$ arbitrarily exceeds $\dim(v)$ \cite{BDV}. Correspondingly, the dense subalgebra of an ergodic action of a finite dimensional quantum group can be of infinite dimension. However, if the action arises from a quantum subgroup $K$ then
$\mult(v)\leq \dim(v)$ for all irreducible representations $v$, hence one still has $\dim(G/K)\leq \dim(G)$.

\section{The Fourier algebra of a compact quantum group}\label{fourier}

The Fourier algebra $A(G)$ in a non-commutative framework has been studied by several authors, see \cite{BS}, for multiplicative unitaries and, more recently, \cite{Caspers, Daws, HuNeuRuan2009,  HuNeuRuan2010, Kahng} for locally compact quantum groups.
In the compact case, discreteness of the dual allows an explicit  description of the  Fourier algebra, and, with a different motivation, Simeng Wang has discussed certain aspects of interest in harmonic analysis \cite{Wang}.

In this section we describe a standard operator algebraic approach to $A(G)$. In particular, we discuss a result about a correspondence between the irreducible representations of $A(G)$ and those of $C_{\text{max}}(G)$. We shall also see that if $G$ is of Kac type then $A(G)$ is an involutive Banach algebra with respect to its natural involution.

\subsection{The Banach algebra $A(G)$}
Consider the dense subalgebra $\Q_G$ of $C(G)$, and express an element $a\in\Q_G$ in the form
$$a=\sum_{v\in\hatG}\sum_{i,j} \lambda^v_{i,j} v_{i,j},$$
where $v_{i,j}$ are coefficients of $v$ with respect to some orthonormal basis. We introduce a new norm in $\Q_G$,
\begin{equation}\label{norm1}
\|a\|_1=\sum_{v\in\hatG}\Tr(|\Lambda_v^t|)
\end{equation}
where $\Lambda_v$ denotes the complex-valued matrix $(\lambda^v_{i,j})$ and $\Tr$ is the non-normalised trace of a matrix algebra. Properties of $\Tr$ imply that $a\mapsto\|a\|_1$ is indeed a norm that depends neither on the choice of the irreducible representations nor on that of the orthonormal bases.    
\begin{thm}\label{Fourier}
The completion $A(G)$ of $\Q_G$ in the norm $a\mapsto\|a\|_1$  is a Banach algebra isometrically isomorphic via the Fourier transform to
$\ell^1(\hatG):=\ell^\infty(\hatG)_*\,\,$, 
where $\hatG$ is the dual quantum group of ${G}$. When $G$ is of Kac type, the natural involution of $\Q_G$ makes $A(G)$ into an involutive Banach algebra.
\end{thm}
 
\begin{defn}
The algebra $A(G)$ is called the {\it Fourier algebra} of $G$.\end{defn}
 
We will now explain the statement of Theorem \ref{Fourier} and sketch the main points in its proof. Recall how the $\ell^1$-algebra of the dual discrete quantum group $\hatG$ is defined. 
Consider the $C^*$--algebra  associated to $\hatG$,
$$c_0(\hatG)=\bigoplus_{v\in\hatG}{\mathcal B}(H_v)$$
and the von Neumann algebra
$$\ell^\infty(\hatG)=\prod_{v\in\hatG} {\mathcal B}(H_v).$$
Duality between $G$ and $\hatG$ in the sense of \cite{BS} is described by the unitary $V\in M(c_0(\hatG)\otimes C(G))$,
$$V=\bigoplus_{v\in\hatG} v,$$
so that the coproduct $\hat{\Delta}:\ell^\infty(\hatG)\to \ell^\infty(\hatG)\otimes \ell^\infty(\hatG)$ is defined by
$$\hat{\Delta}\otimes 1(V)=V_{13}V_{23}.$$
Explicitly,
\begin{equation}\label{Deltahat}
\hat{\Delta}(\Theta^v_{\xi,\eta})=\sum_{u,u'\in\hatG}\sum_p\Theta^{u\otimes u'}_{S_p\xi, S_p\eta},
\end{equation}
where $S_p\in(v, u\otimes u')$ is a maximal family of isometric intertwiners with mutually orthogonal ranges. The left- and right-invariant Haar weights $\hat{h}_l$, $\hat{h}_r$ of $\ell^\infty(\hatG)$ are respectively given by
$$\hat{h}_l(T)=\sum_{v\in\hatG}\qdim(v)\Tr(F_v^{-1}T_v), \quad\hat{h}_r(T)=\sum_{v\in\hatG}\qdim(v)\Tr(F_vT_v),\quad T\in\ell^\infty(\hatG)^+.$$
The GNS representation associated with $\hat{h}_l$ provides an action of $\ell^\infty(\hatG)$ on the Hilbert space $\ell^2(\hatG)$. Set
$\ell^1(\hatG)=\ell^\infty(\hatG)_*.$
The coproduct of $\ell^\infty(\hatG)$ is a normal faithful $^*$--homomorphism, which induces a contractive map
$$\hat{\Delta}_*: \ell^1(\hatG)\otimes\ell^1(\hatG)\to\ell^1(\hatG)$$
making $\ell^1(\hatG)$ into a Banach algebra with respect to the convolution product given by
\begin{equation}\label{convolution1}
\omega*\omega' = \hat{\Delta}_*(\omega\otimes\omega')=\omega\otimes\omega'\circ\hat{\Delta}.
\end{equation}
There is a natural isometric identification of Banach spaces
$$\{A\in\prod_{v\in\hatG} {\mathcal B}(H_v): \|A\|_1= \sum_{v\in\hatG}\qdim(v)\Tr(|A_vF_v^{-1}|)<\infty\}\to\ell^1(\hatG)$$
taking $A$ to the functional $\omega_A\in\ell^1(\hatG)$ given by
\begin{equation}\label{omegaA}
\omega_A(T):=\hat{h}_l(TA), \quad\quad\quad T\in\ell^\infty(\hatG).
\end{equation}
We henceforth realise $\ell^1(\hatG)$ in this way. The algebraic direct sum
$${\mathcal P}_{\hatG}= \bigoplus^{\text{alg}}_{v\in\hatG}{\mathcal B}(H_v)\subset\ell^1(\hatG)$$ then becomes a dense subalgebra of
$\ell^1(\hatG)$ with respect to the convolution product 
\begin{equation}\label{convolution2}
\Theta^u_{\eta,\xi}*\Theta^{u'}_{\eta', \xi'}:=\sum_{w\in\hatG}\sum_i \frac{\qdim(u)\qdim(u')}{\qdim(w)}\Theta^w_{S^*_{w,i}(\eta\eta'), S^*_{w,i}(\xi\xi')},
\end{equation}
where $S_{w,i}\in (w, u\otimes u')$ is a complete set of isometries with mutually orthogonal ranges. 
This formula can be derived from the identification \eqref{omegaA}, and relations \eqref{Deltahat}, \eqref{convolution1}. The algebra unit is the identity operator on the space of the trivial representation.
 
We next construct the Fourier algebra $A(G)$ generalising the classical approach from \cite{HR}.
For $a\in C(G)$ we define the {\it Fourier coefficients} by
$$\hat{a}(v)=1_{{\mathcal B}(H_v)}\otimes h(v^* (I_{{\mathcal B}(H_v)}\otimes a))\in{\mathcal B}(H_v), \quad v\in\hatG.$$
The orthogonality relations \eqref{orth1} then imply
\begin{equation}\label{uhat}
\hat{u}_{\xi,\eta}(v)=\frac{1}{\qdim(v)}\delta_{u,v}\Theta_{\eta,\xi}F_v.
\end{equation}
Then the following {\it Fourier inversion formula} holds:
$$a=\sum_{v\in\hatG}\qdim(v)\Tr\otimes 1_{\Q_G}(((\hat{a}(v)F_v^{-1})\otimes I)v), \quad a\in\Q_G.$$
 
\begin{prop}\label{isomorphism}
The Fourier transform ${\mathcal F}: a\in\Q_G\to\hat{a}\in{\mathcal P}_{\hatG}$ is an algebra isomorphism which extends to an isometric isomorphism of Banach spaces
$${\mathcal F}: A(G)\to \ell^1({\hatG}).$$
This makes $A(G)$ into a Banach algebra.
\end{prop}
 
\begin{proof}
We may write, for $u$, $u'\in\hatG$, and $S_{w,i}\in(w, u\otimes u')$ as before,
$$u_{\xi,\eta}u'_{\xi', \eta'}=(u\otimes u')_{\xi\xi', \eta\eta'}=\sum_{w,i}w_{{S_{w,i}^*}\xi\xi', {S_{w,i}^*}\eta\eta'}.$$
Hence
$${\mathcal F}({ u_{\xi,\eta}u'_{\xi', \eta'}})(v)=\sum_{w,i}{\mathcal F}({w_{{S_{w,i}^*}\xi\xi', {S_{w,i}^*}\eta\eta'}})(v)=
\frac{1}{\qdim(v)}\sum_i\Theta^v_{{S_{v,i}^*}\eta\eta', {S_{v,i}^*}\xi\xi'} F_{v}.$$
On the other hand, by $(5.5)$ we have
$${\mathcal F}({u_{\xi,\eta}})*{\mathcal F}( {u'_{\xi',\eta'}})(v)=\frac{1}{\qdim(v)}\sum_i\Theta^v_{S^*_{v,i}\eta\eta', S^*_{v,i}F_u\otimes F_u'\xi\xi'},$$
hence the two expressions coincide, thanks to 
\begin{equation}\label{SFv}
S F_v=F_u\otimes F_{u'}S,\quad\quad\quad S\in(v, u\otimes u').
\end{equation}
We next notice that the trace norm of $\Q_G$ defined in \eqref{norm1} is but the norm making ${\mathcal F}$ isometric. In particular, $\Q_G$ becomes a normed algebra. The remaining statement is now clear.
\end{proof}
 
There are two interesting natural involutions on $A(G)$: the first arises from the original involution $a\mapsto a^*$ of $\Q_G$, and is useful when dealing with  
the maximal  $C^*$--completion $C_{\text{max}}(G)$, 
but is in general only densely defined in $A(G)$; the second is given by the polar decomposition of the former, and its importance is due to the fact that it makes $A(G)$ into an involutive Banach algebra. These two involutions coincide precisely when $G$ is of Kac type.
 We study the former involution here, and postpone the latter to Sect. \ref{nonkac}.
\bigskip

\subsection{The $^*$--involution}
In this subsection $A(G)$ is regarded as a Banach algebra endowed with the densely defined involution $a\in \Q_G\mapsto a^*\in\Q_G$.

\begin{defn} 
A $^*$--representation of $A(G)$ is a Hilbert space representation of $A(G)$ which is a $^*$--representation of $\Q_G$. 
\end{defn}

We start by recalling a continuity result of $^*$--preserving Hilbert space representations of $\Q_G$ which can be used to continuously embed the Fourier algebra $A(G)$ into $C(G)$.
\begin{prop}\label{estimate}
Let $G$ be a compact quantum group. Every $^*$--representation $\pi$ of $\Q_G$ on a Hilbert space $\H$ satisfies 
$$\|\pi(a)\|_{\B(\H)}\leq \|a\|_1,\quad\quad\quad a\in\Q_G,$$
hence it extends to a contractive $^*$-representation of $A(G)$. 
\end{prop}
\begin{proof}
A slight modification of the argument in \cite[Prop. 2.9]{Wang2} proves the result in the more general setting that we are considering. Namely, it suffices to replace the predual of $L^\infty(G)$ with that of $\B(\H)$ (or just the dual Banach space) and the norm of $L^\infty(G)$ with that of $\B(\H)$, and notice that the same computations involving the Fourier inversion formula go through since $\|1\otimes\pi(u)\|\leq 1$ by unitarity of $u$ and the fact that $\pi$ is $^*$--preserving.
\end{proof}

\begin{cor}\label{iota}
The natural inclusion $\Q_G\subset C(G)$ extends to a contractive inclusion
$$\iota: A(G)\to C(G)$$ of Banach algebras. 
\end{cor}

\begin{proof}
Apply Prop. \ref{estimate} to a faithful Hilbert space realisation of $C(G)$.
\end{proof}

We next show that $A(G)$ is a semisimple Banach algebra. Consider  the GNS representation $(L^2(G), \pi_h)$ of $C(G)$ associated with the Haar state $h$ of $G$, and restrict it to a $^*$--representation $\pi$ of $A(G)$ via $\pi=\pi_h\circ \iota$. The following fact is standard.
\begin{prop}\label{contractive}
The Fourier transform extends to a unitary operator
$$U_{\mathcal F}: L^2(G)\to\ell^2(\hatG).$$
Furthermore, the action of ${\mathcal P}_{\hatG}$ on itself by convolution extends to a contractive representation $\lambda$ of $\ell^1(\hatG)$ on $\ell^2(\hatG)$ and one has
\begin{equation}\label{UFpi}
U_{\mathcal F}\pi(x)=\lambda({\mathcal F}(x))U_{\mathcal F},\quad\quad\quad x\in A(G).
\end{equation}
\end{prop}
\begin{proof}
Unitarity of $U_{\mathcal F}$ is an immediate consequence of \eqref{uhat} along with the orthogonality relations \eqref{orth1}. By Prop. \ref{isomorphism}, the intertwining relation \eqref{UFpi} holds for $x\in{\mathcal P}_{\hatG}$ on a dense subspace of $L^2(G)$.
Thus $\lambda(y)$ is a bounded operator on $\ell^2(\hatG)$ for $y\in{\mathcal P}_{\hatG}$ and $\|\lambda(y)\|\leq\|{\mathcal F}^{-1}(y)\|_1=\|y\|_1$. We conclude that $\lambda$ extends to a bounded representation of $\ell^1(\hatG)$ and \eqref{UFpi} still holds for the extension.
\end{proof}

There is an antilinear involution on ${\mathcal P}_{\hatG}$ given by
$$(\Theta^u_{\eta,\xi})^*=\Theta^{{\overline{u}}}_{ {j_v^*}^{-1}\eta, {j_v^*}^{-1}\xi}F_{\overline{v}}.$$
This coincides with the involution inherited from $\Q_G$ via Fourier transform. Indeed, we recall from Subsect. \ref{PolyGK}, that if $j: H_u\to H_{\overline{u}}$ defines a conjugate of $u$ then
$u_{\xi,\eta}^*=\overline{u}_{j\xi, {j^*}^{-1}\eta}$,  $F_u= j^*j$, $F_{\overline{u}}=(jj^*)^{-1}$.
Hence
$${\mathcal F}(u_{\xi,\eta}^*)={\mathcal F}(\overline{u}_{j\xi, {j^*}^{-1}\eta})=\frac{1}{\qdim(\overline{u})}\Theta^{\overline{u}}_{{j^*}^{-1}\eta, j\xi}F_{\overline{u}}=\frac{1}{\qdim(\overline{u})}\Theta^{\overline{u}}_{{j^*}^{-1}\eta, {j^*}^{-1}\xi},$$
and this equals ${\mathcal F}(u_{\xi, \eta})^*$. This involution coincides also with that inherited from the Hilbert space representation $\lambda$.

\begin{cor}\label{faithful}
The $^*$--representation $\pi=\pi_h\circ\iota$ of $A(G)$ is faithful. In other words, $A(G)$ is semisimple.
\end{cor}

\begin{proof}
Prop. \ref{contractive} shows that the operator $\ell^1(\hatG)\mapsto \ell^2(\hatG)$, $x\to\lambda(x)\eta$, is contractive, where $\eta\in \ell^2(\hatG)$ is a normalized vector supported on the trivial representation. On the other hand this operator acts trivially on $\ell^1(\hatG)$, hence $\ell^1(\hatG)\subset\ell^2(\hatG)$.
If $\pi(x)=0$ for some $x\in A(G)$ then $\lambda({\mathcal F}(x))\eta=0$ by \eqref{UFpi}. Hence ${\mathcal F}(x)=0$ and this implies $x=0$.
\end{proof}

We denote by $\widehat{A(G)}$ the set of equivalence classes of topologically irreducible 
$^*$--representations of $A(G)$. Let $C^*(A(G))$ be the completion of $A(G)$ in the norm
$$\|x\|_{\text{max}}=\text{sup}_{\pi\in\widehat{A(G)}}\|\pi(x)\|,\quad\quad\quad x\in A(G).$$ 
The natural map
$A(G)\to C^*(A(G))$
is faithful and contractive.

\begin{thm}
$C^*(A(G))$ is a $C^*$--algebra and coincides with ${C_{\text{\rm max}}(G)}$.
\end{thm}
\begin{proof}
Since $\|x\|_{\text{max}}\leq \|x\|_1$, $\Q_G$ is dense in $C^*(A(G))$. Being a $C^*$--completion of a $^*$--algebra, $C^*(A(G))$ is a $C^*$--algebra. On the other hand, the map $\pi\to\pi\circ\iota$, with $\iota$ defined as in the Cor. \ref{iota}, establishes a bijective correspondence between $\widehat{C_{\text{\rm max}}(G)}$ and $\widehat{A(G)}$ by Prop. \ref{estimate}. We conclude that the norm of $C^*(A(G))$ equals the maximal norm of $C_{\text{max}}(G)$ on $\Q_G$.
\end{proof}
 
\section{Polynomial growth and $^*$--regularity}\label{growth}
 
If $G$ is a coamenable compact quantum group, the dense Hopf subalgebra $\Q_G$ admits a unique $C^*$--completion to a compact quantum group, since $C_{\rm max}(G)=C_{\rm red}(G)=C(G)$.
However, $\Q_G$ in general does not determine ${C(G)}$ as a $C^*$--algebra, as it often admits several $C^*$--norms. This can be seen already for the circle group, $G={\mathbb T}$, where the supremum norm on any infinite closed subset $C\subset {\mathbb T}$ gives a $C^*$--norm due to the fact that elements of $\Q_G$ are restrictions to the torus of analytic functions.
We thus need to replace $\Q_G$ by a larger $^*$--algebra. 
One of the results of this section is that for compact quantum groups of Kac type and of polynomial growth the Fourier algebra $A(G)$, when regarded as a subalgebra of $C(G)$, is the correct algebra, in that it does have a unique $C^*$--norm. As mentioned in the introduction, this is related to previous work on $^*$-regularity dating back to the '80s.

Recall that if $A$ is a (semisimple) Banach $^*$--algebra, the spectrum $\widehat{A}$ is the set of equivalence classes of topologically irreducible $^*$--representations of $A$ on Hilbert spaces. This is a $T_0$-topological space with the Jacobson topology, defined as follows. Let $\text{Prim}(A)$ denote the space of kernels of elements of $\widehat{A}$ endowed with the hull-kernel topology. We have a natural map
$$\kappa: \widehat{A}\to \text{Prim}(A)$$
associating a representation with its kernel, which is always surjective but may fail to be injective. For instance, if $A$ is a $C^*$--algebra, this map is injective if and only if $A$ is of type $I$. 
The Jacobson topology on $\widehat{A}$ is the weakest topology making $\kappa$ continuous.

In the framework of compact quantum groups, when $G={\rm SU}_q(d)$, it is known that $C(G)$ is of type $I$, and the Jacobson topology of $\widehat{C(G)}$ has been described, see \cite{NT} and references therein.
Consider, for a general Banach $^*$--algebra $A$, the continuous surjective map
$$\Psi_A: \text{Prim}(C^*(A))\ni \ker \pi \mapsto \ker \pi \cap A \in \text{Prim}(A).$$
\begin{defn}
$A$ is called $^*$--regular if $\Psi_A$ is a homeomorphism.
\end{defn}
    
Clearly, $\widehat{A}$ is in bijective correspondence with $\widehat{C^*(A)}$.  If $A$ is   $^*$--regular then the identification holds also at the level of topological spaces. In particular, if $G$ is a compact quantum group of Kac type, $^*$--regularity of $A(G)$ ensures that $\widehat{C_{\text{max}}(G)}$ identifies topologically with $\widehat{A(G)}$.

There are several statements equivalent to $^*$--regularity, such as asking that ${\mathcal I}\cap A$ be dense in ${\mathcal I}$ for every closed ideal of $C^*(A)$. The notion of $^*$--regularity is closely related to uniqueness of a $C^*$--norm. It is known that $A$ has a unique $C^*$--norm if and only if ${\mathcal I}\cap A\neq0$ for every nonzero ideal ${\mathcal I}$ as above and this implies that $^*$--regular Banach algebras have a unique $C^*$--norm.
Furthermore, $A$ is $^*$--regular if and only if all quotients $A/A\cap{\mathcal I}$ have a unique $C^*$--norm  \cite{Barnes83}.

The property of being $^*$--regular was first studied for the group algebra $L^1(\Gamma)$ of a locally compact group \cite{BLSV}. In particular, the authors show that if the Haar measure of $\Gamma$ has polynomial growth (that is for every compact subset $K\subset \Gamma$, $\mu(K^n)=O(n^N)$ for some integer $N$) then $L^1(\Gamma)$ is $^*$--regular.
The aim of this section is to show a non-commutative analogue of this result.
\begin{thm}\label{main}
Let $G$ be a compact quantum group of Kac type. If $\Q_G$ is of polynomial growth, then $A(G)$ is $^*$--regular.
\end{thm}

We will prove this by translating \cite{BLSV} so as to apply it to the Fourier algebra of a compact quantum group.
An important aspect of the original proof is the construction of a functional calculus for a dense subset of elements, which can be traced back to the work of Dixmier \cite{Dixmier} on nilpotent Lie groups. We extend these ideas to general Banach $^*$--algebras.

Let $A$ be a Banach algebra. Set $\tilde{A}=A$ if $A$ is unital and $\tilde{A}=A\oplus {\mathbb C}I$ otherwise. Given an element  $f\in A$, define
$$e^{if}=\sum_{k=0}^\infty \frac{(if)^k}{k!}\in\tilde{A}.$$
We shall say that $f$ has polynomial growth if there exists $\gamma>0$ such that
$$\|e^{i\lambda f}\|=O(|\lambda|^\gamma)\quad\quad\quad\text{for } |\lambda|\to+\infty.$$
The following lemma is a well-known  abstract reformulation of \cite[Lemma 7]{Dixmier}.

\begin{lemma} Let $A$ be a Banach $^*$--algebra. For any $C^\infty$-function $\varphi:{\mathbb R}\to{\mathbb C}$ with compact support and any element $f\in A$ of polynomial growth,  
\begin{itemize}
\item[{\rm a)}]
the integral
$$\varphi\{f\}:=\frac{1}{2\pi}\int_{-\infty}^{+\infty} e^{i\lambda f}\hat{\varphi}(\lambda) d\lambda$$
is absolutely convergent in $\tilde{A}$,
\item[{\rm b)}]
if $A$ is non-unital, $\varphi\{f\}\in A$ whenever $\varphi(0)=0$, 
\item[{\rm c)}]
for every $^*$--representation $\pi$ of $A$,
$$\pi(\varphi\{f\})=\varphi(\pi(f))$$
if $f=f^*$, where the right-hand side denotes the continuous functional calculus of the operator $\pi(f)$. 
\end{itemize}
\end{lemma}
We now give a straightforward abstraction of the main argument of   \cite{BLSV}.
\begin{lemma}
Let $A$ be a Banach $^*$--algebra admitting a subset of elements of polynomial growth dense in $A_{sa}$. Then $A$ is $^*$--regular.
\end{lemma}
\begin{proof}
By \cite[Prop. 1]{BLSV}, see also \cite[Prop. 1.3]{LN}, we need to show that $\|\rho(f)\|\leq\|\pi(f)\|$ holds for all $f\in A$ whenever $\pi, \rho$ are $^*$--representations of $A$ satisfying $\text{ker} \pi\subset \text{ker}\rho$.

By the $C^*$--property it suffices to show this for all selfadjoint elements $f$, and by our assumption we may also assume $f$ to be of polynomial growth. Assume on the contrary there exists such an $f\in A_{sa}$ with $\|\pi(f)\|<\|\rho(f)\|$. Let $\varphi$ be a positive $C^\infty$-function with compact support such that $\varphi(x)=0$ for $|x|\leq\|\pi(f)\|$ and $\varphi(\pm\|\rho(f)\|)=1$. Then $\varphi$ vanishes on $\text{Sp}\pi(f)$ and $\text{sup}\{\varphi(t), t\in\text{Sp}\rho(f)\}\geq1$ since $\pi(f)$ und $\rho(f)$ are selfadjoint operators. By the previous lemma,
$$\|\pi(\varphi\{f\})\|=\|\varphi(\pi(f))\|=\text{sup}\{\varphi(t); t\in\text{Sp}\,\pi(f)\}=0$$
and
$$\|\rho(\varphi\{f\})\|=\|\varphi(\rho(f))\|=\text{sup}\{\varphi(t); t\in\text{Sp}\,\rho(f)\}\geq 1.$$
This contradicts our assumption that $\text{ker}\pi\subset\text{ker}\rho$. 
\end{proof}

Dixmier showed that if $\Gamma$ is a unimodular locally compact group with Haar measure of polynomial growth then every continuous function $f\in C_c(\Gamma)$ with compact support has polynomial growth in $\widetilde{L^1(\Gamma)}$.
The following lemma is a non-commutative analogue.
 
\begin{lemma}
Let $G$ be a compact quantum group of Kac type. If    $\Q_G$ is of polynomial growth, then every selfadjoint $f\in\Q_G$ has polynomial growth in $A(G)$.
\end{lemma}
\begin{proof}
By Prop. \ref{isomorphism} and the remarks preceding Corollary \ref{faithful}, it is enough to prove the statement for a selfadjoint element $g\in{\mathcal P}_{\hatG}$ and the Banach algebra $\ell^1({\hatG})$.

We adapt the proof of \cite[Lemmas 5, 6]{Dixmier} by replacing the role of a finite measure subset $A$ of $\Gamma$ with a finite set $F\subset\hatG$ of irreducible representations of $G$, and powers $A^p$ with the set of irreducible representations contained in $u_F^{\otimes p}$, where $u_F$ is the direct sum of elements of $F$. Choose now $F$ so that it contains the support of $g$. Finally, perform the same computations where the norms of $L^1(\Gamma)$ and $L^2(\Gamma)$ are replaced by the norms of $\ell^1(\hatG)$,
$\ell^2(\hatG)$ mentioned in the previous section.
\end{proof}

The proof of Theorem \ref{main} is now complete.

\begin{rem}
Examples of Kac-type compact quantum groups that are not $^*$--regular are provided by non-amenable discrete groups. Indeed, for any such $\Gamma$, $\ell^1(\Gamma)$ is not $C^*$--unique. As a matter of fact, it is not as easy to find examples among amenable \emph{discrete} groups. Indeed, on the one hand the relationship between $^*$--regularity and $C^*$--uniqueness has been studied at an in-depth level in  \cite{Barnes83}.
On the other, it is not known whether every such group is automatically $C^*$--unique, see \cite{LeungNg} for partial positive results.
The most natural candidate to disprove this conjecture is the so-called Grigorchuk group as remarked in \cite{GMR}. 
Moreover, giving workable examples of coamenable Kac-type compact quantum groups, beyond the cocommutative or commutative cases, is not easy either, see  the recent paper \cite{NY} and references therein.
\end{rem}

\section{Remarks on the non-Kac-type case}\label{nonkac}

We have already noticed that $A(G)$ is no longer a natural involutive algebra in the non-Kac-type case. Nonetheless, $A(G)$ can still be made into an involutive Banach algebra, as we next explain. 
\bigskip

\subsection{The $^\bullet$-involution}

Taking into account the modular structure of $\hatG$, the new involution is defined by
$$a^{\bullet}=(\tau(a))^*, \quad\quad\quad a\in\Q_G,$$
where $\tau$ is the automorphism of $\Q_G$ given by
$$\tau: u_{\xi,\eta}\in\Q_G\to u_{F^{-1/2}\xi, F^{1/2}\eta}\in \Q_G,\quad\quad\quad u\in\hatG.$$
This is of course the Woronowicz automorphism arising from the polar decomposition of the antipode of $\Q_G$. Thus $a^\bullet=a^*$ precisely when $G$ is of Kac type.
  
\begin{thm}
The involution $a\to a^\bullet$ makes $A(G)$ into an involutive Banach algebra.
\end{thm}
\begin{proof}
Recall from Subsect. \ref{PolyGK} that  if $j_u: H_u\to H_{\overline{u}}$ defines a conjugate of the irreducible representation $u$ then $u_{\xi,\eta}^*=\overline{u}_{j\xi, {j^*}^{-1}\eta}$. Let $j_u=J_uF_u^{1/2}$ be the polar decomposition of $j_u$. Then $J_u: H_u\to H_{\overline{u}}$ is antiunitary and we have
$$u_{\xi,\eta}^\bullet=\overline{u}_{J_u\xi, J_u\eta}, \qquad J_{\overline{u}}=J_u^{-1},$$
hence $a^{\bullet\bullet}=a$ for $a\in\Q_G$. Explicit computations show that, if $a\in\Q_G$ is determined by the matrices $\Lambda_v=(\lambda^v_{i,j})$, that is $a=\sum_v \lambda^v_{i,j} u_{i, j}$, $a^\bullet$ is then determined by the matrices $\tilde{\Lambda}_v$ where $\tilde{\Lambda}_v^t=J_v\Lambda_v^tJ_v^{*}$.
Hence 
$|\tilde{\Lambda}_v^t|=J_v|\Lambda_v^t|J_v^{*}$ and this, together with antiunitarity of $J_v$, implies $\|a^\bullet\|_1= \|a\|_1$.
\end{proof}

Investigating semisimplicity, $^*$--regularity or $C^*$--uniqueness of the $^*$-algebra $A(G)$ seems an interesting problem. Yet, $A(G)$
seems only loosely related to topological dimension. More precisely, the study of the relationship between dimension and $^*$--regularity for ${\rm SU}_q(2)$  might involve a quantum analogue of the Beurling-Fourier algebra \cite{LS, LST} with weights of exponential type, of which little is known already in the case of classical compact groups. These issues are beyond the scope of the present paper and will be dealt with elsewhere.
 
\bigskip

\noindent{\it Acknowledgements} We would like to thank Mikael R\o rdam for a conversation related to the upcoming paper \cite{GMR}.
We acknowledge support by AST fundings from ``La Sapienza'' University of Rome.

\vfill

\bigskip 


\begin{thebibliography}{DLM}
\bibitem{BS} S. Baaj, and G. Skandalis, \emph{Unitaires multiplicatifs et dualit\'e pour les produits crois\'es de ${C}^* $--alg\`ebres}, Ann. Sci. \'Ec. Norm. Sup\'er. \textbf{26} 1993, 425--488.
\bibitem{Banica99} T. Banica, \emph{Representations of compact quantum groups and subfactors}, J. reine angew. Math. \textbf{509} (1999), 167--198.
\bibitem{BV}  T. Banica, and R. Vergnioux, \emph{Growth estimates for discrete quantum groups}, Infin. Dimens. Anal. Quantum Probab. Relat. Top. \textbf{12} (2009), 321--340. 
\bibitem{Barnes81} B. A. Barnes, \emph{Ideal and representation theory of the $L^1$--algebra of a group with polynomial growth}, Coll. Math. \text{45} (1981),  301--315.
\bibitem{Barnes83} B. A. Barnes, \emph{The properties $^*$--regularity and rniqueness of $^*$--norm in a general  $^*$--algebra}, Trans. Amer. Math. Soc. \textbf{279} (1983),  841--859.
\bibitem{BDV} J. Bichon, A. De Rijdt, and S. Vaes, \emph{Ergodic coactions with large multiplicity and monoidal equivalence of quantum groups}, Comm. Math. Phys. \textbf{262} (2006), 703--728.
\bibitem{BLSV} J. Boidol, H. Leptin, J. Sch{\" u}rman, and D. Vahle, \emph{R{\" a}ume primitiver Ideale von Gruppenalgebren}, Math. Ann. \textbf{236} (1978), 1--13.
\bibitem{Boidol84} J. Boidol, \emph{Group algebras with a unique $C^*$--norm}, J. Funct. Anal. \textbf{56} (1984), 220--232.
\bibitem{BTD} T. Br{\"o}cker, and T. tom Dieck, Representations of compact Lie groups, 1 Springer, 1985.
\bibitem{Caspers} M. Caspers, \emph{The $L^p$--Fourier transform on locally compact quantum groups}, J. Op. Theory \textbf{69}  (2013), 161--193.
\bibitem{Daws} M. Daws, \emph{Multipliers of locally compact quantum groups via Hilbert $C^*$--modules}, J. Lond. Math. Soc. \textbf{84} (2011), 385--407.
\bibitem{Dixmier} J. Dixmier, \emph{Op\'{e}rateurs de rang fini dans les repr\'{e}sentation unitaires}, Publ.  Math. I.H.E.S,   \textbf{6} (1960), 13--25.
\bibitem{GK} I.M. Gelfand, and A. Kirillov, \emph{Sur les corps li\'es aux alg\`ebres enveloppantes des alg\`ebres de Lie}, Publ.  Math. I.H.E.S,  \textbf{31} (1966), 5--19.
\bibitem{GMR} R. Grigorchuk, M. Musat, and M. R\o rdam,      cf. M. R\o rdam: \emph{Just infinite groups and $C^*$--algebras}, slides of  talk, COSy, Waterloo, June  2015.
\bibitem{HR} E. Hewitt, and K. A. Ross, Abstract harmonic analysis: Volume II: 
Structure and Analysis for Compact Groups Analysis on Locally Compact Abelian Groups. Vol. 152. Springer, 2013.
\bibitem{HuNeuRuan2009} Z. Hu, M. Neufang, and Z. J. Ruan, \emph{On topological centre problems and SIN quantum groups}, J. Funct. Anal. \textbf{257} (2009), 610--640.
\bibitem{HuNeuRuan2010} Z. Hu, M. Neufang, and Z. J. Ruan, \emph{Multipliers on a new class of Banach algebras, locally compact quantum groups and topological centres}, Proc. London Math. Soc. \textbf{100} (2010), 429--458.
\bibitem{Kahng} B. J. Kahng, \emph{Fourier transform on locally compact quantum groups}, J. Op. Theory \textbf{64}  (2010), 69--87. 
\bibitem{Krause_Lenagan} G.R. Krause, and T.H. Lenagan, Growth of algebras and Gelfand-Kirillov dimension, Graduate studies in Mathematics, vol. 22,  AMS, 1999.
\bibitem{LeungNg} C. W. Leung, and  C.K. Ng, \emph{Some permanence properties of $C^*$--unique groups}, J. Funct. Anal. \textbf{210}   (2004), 376--390.
\bibitem{LN} C. W. Leung, and  C.K. Ng, \emph{Functional calculus and $^*$--regularity of a class of Banach algebras}, Proceed. Amer. Math. Soc. \textbf{134}  (2005), 755--763.
\bibitem{LS} H. H. Lee, and E. Samei, \emph{Beurling-Fourier algebras, operator amenability and Arens regularity}, J.  Funct. Anal. \textbf{262} (2012), 167--209.
\bibitem{LST} J. Ludwig, N. Spronk, and L. Turowska, \emph{Beurling-Fourier algebras on compact groups: Spectral theory}, J.  Funct. Anal. \textbf{262} (2012) 463--499.
\bibitem{MaesVanDaele} A. Maes, and A. Van Daele, \emph{Notes on compact quantum groups}, Nieuw Arch. Wisk.  \textbf{16} (1998) 73--112.
\bibitem{NT} S. Neshveyev, and L. Tuset, \emph{Quantized algebras of functions on homogeneous spaces with Poisson stabilizers},  Comm. Math. Phys. \textbf{312} (2012),  223--250.
\bibitem{NTbook} S. Neshveyev, and L. Tuset, Compact quantum groups and their representation categories, AMS Cours Sp\'{e}cialis\'{e}s \textbf{20} 2014.
\bibitem{NY} S. Neshveyev, and M. Yamashita, \emph{Classification of non-Kac compact quantum groups of ${\rm SU}(n)$ type}, Int. Math. Res. Notices (2015), rnv241.  
\bibitem{PR} C. Pinzari, and J. E. Roberts, \emph{A duality theorem for ergodic actions of compact quantum groups on $C^*$--algebras}, Comm.   Math. Phys. \textbf{277} (2008), 385--421.
\bibitem{Podles} P. Podles, \emph{Symmetries of quantum spaces. Subgroups and quotient spaces of quantum ${\rm SU}(2)$ and 
${\rm SO}(3)$ groups} Comm. Math. Phys., \textbf{170} (1995), 1--20.
\bibitem{Takahashi} K. Takahashi, \emph{Dimension of compact groups and their representations}, Tohoku Math. Journ. \textbf{23} (1971), 663--670.
 \bibitem{Vergnioux} R. Vergnioux, \emph{The property of rapid decay for discrete quantum groups},  J.  Op. Theory \textbf{57} (2007),  303--324.
\bibitem{Wang} S. Wang, \emph{$L^p$--Improving convolution operators on finite quantum groups},   arXiv:1412.2085. 
\bibitem{Wang2} S. Wang, \emph{Lacunary Fourier series for compact quantum groups}, arXiv:1507.01219.
\bibitem{WoroCMP} S. L. Woronowicz, \emph{Compact matrix pseudogroups}, Comm. Math. Phys. \textbf{111}  (1987), 613--665.
\bibitem{cqg} S. L. Woronowicz, \emph{Compact quantum groups},  Sym\'etries quantiques (Les Houches, 1995) \textbf{845} (1998).
\bibitem{Zhuang} G. Zhuang, \emph{Properties of pointed and connected Hopf algebras of finite Gelfand-Kirillov dimension}, J. Lond. Math. Soc. \textbf{87} (2013), 877--898.




 







 \end{thebibliography}
\end{document}